\documentclass{amsart}
\usepackage{amssymb,amsmath,amsrefs}
\usepackage[colorlinks=true]{hyperref}
\hypersetup{urlcolor=blue, citecolor=green}
\usepackage[dvips]{graphicx}
\usepackage{color}

\theoremstyle{plain}

\newtheorem{mainthm}{Theorem}

\newtheorem*{conj*}{Conjecture}
\newtheorem*{cor*}{Corollary}
\newtheorem{theorem}{Theorem}[section]

\newtheorem{proposition}[theorem]{Proposition}
\newtheorem{corollary}[theorem]{Corollary}

\newtheorem{question}{Question}

\newtheorem*{def*}{Definition}

\newtheorem{example}[theorem]{Example}

\newtheorem{definition}[theorem]{Definition}

\renewcommand{\epsilon}{\varepsilon}

\newcommand{\Z}{\mathbb{Z}}
\newcommand{\N}{\mathbb{N}}

\newcommand{\eps}{\varepsilon}

\newcommand{\diam}{\operatorname{diam}}

\title[Sensitivity, local stable/unstable sets and shadowing]{Sensitivity, local stable/unstable sets and shadowing}

\author[Mayara Antunes, Bernardo Carvalho and Margoth Tacuri]{Mayara Antunes, Bernardo Carvalho and Margoth Tacuri}
\date{\today}


\begin{document}
\begin{abstract}
In this paper we study local stable/unstable sets of sensitive homeomorphisms with the shadowing property defined on compact metric spaces. We prove that local stable/unstable sets always contain a compact and perfect subset of the space. As a corollary we generalize results in \cite{ACCV} and \cite{CC2} proving that positively countably expansive homeomorphisms defined on compact metric spaces satisfying either transitivity and the shadowing property, or the L-shadowing property, can only be defined in countably spaces.
\end{abstract}

\maketitle

\section{Introduction and the main result}

The concept of chaos is in the core of the dynamical systems theory. The understanding that random behavior can occur from the evolution of deterministic systems, such as in discrete dynamical systems on compact metric spaces, lead many mathematicians to try to formalize the notion of chaos. The first to do that, to our best knowledge, was Guckenheimer \cite{G} in the setting of one dimensional maps. R. Devaney, in his book \cite{D}, gathered some of these attempts in a definition that is now known as Devaney chaotic systems. The dynamical property that captures the central idea of chaos is the \emph{sensitivity to initial conditions}. This is best illustrated by Edward Lorenz and his ideas of the instability of the atmosphere and the butterfly effect \cite{L}. We now define it precisely.

\begin{definition}
A map $f:X\rightarrow X$ defined in a compact metric space $(X,d)$ is \emph{sensitive} if there is $\varepsilon>0$ such that for every $x\in X$ and every $\delta>0$ there exist $y\in X$ with $d(x,y)<\delta$ and $n\in\mathbb{N}$ satisfying $$d(f^{n}(x),f^{n}(y))>\varepsilon.$$ The number $\eps$ is called the \emph{sensitivity constant} of $f$. 
\end{definition}

Sensitivity means that for each initial condition there are arbitrarily close distinct initial conditions having completely different futures. We can explain sensitivity in a few distinct ways. Denoting by $$B(x,\delta)=\{y\in X; \,\,\, d(y,x)<\delta\}$$ the ball centered at $x$ and radius $\delta$, sensitivity implies the existence of $\eps>0$ such that for every ball $B(x,\delta)$ with $x\in X$ and $\delta>0$, there exists $n\in\N$ such that 
$$\diam(f^n(B(x,\delta)))>\eps$$ 
where $\diam(A)=\sup\{d(a,b); a,b\in A\}$ denotes the diameter of $A$. Thus, sensitivity increases the diameter of balls with positive radius. We can explain sensitivity using local stable sets as follows. We define the $\varepsilon$\emph{-stable set} of $x$ by 
$$W^s_\varepsilon(x):=\{y\in X\;;\;d(f^n(y), f^n(x))\leq\varepsilon \,\,\,\,\,\, \text{for every} \,\,\,\,\,\, n\in\mathbb{N}\}.$$ Roughly speaking, this is the set of initial conditions whose futures are similar to the future of $x$. It follows that $f$ is sensitive if, and only if, there exists $\eps>0$ such that for any $x\in X$, the $W^s_{\eps}(x)$ does not contain any neighborhood of $x$. Thus, sensitivity can be seen as a condition on all local stable sets of the space. 
In this paper we study sensitivity for homeomorphisms and how it implies the existence of several initial conditions with similar pasts. The idea is to understand whether sensitivity can also be seen as a condition on all local unstable sets, ensuring they are non-trivial in distinct scenarios. Recall the definition of the $\varepsilon$\emph{-unstable set} of $x$: $$W^u_\varepsilon(x):=\{y\in X\;;\;d(f^{-n}(y), f^{-n}(x))\leq\varepsilon \,\,\,\,\,\, \text{for every} \,\,\,\,\,\, n\in\mathbb{N}\}.$$ Analogously, this is the set of initial conditions whose pasts are similar to the past of $x$. The main theorem of this paper proves the existence of a compact and perfect subset of any local unstable set for assuming sensitivity and the shadowing property. 

\begin{definition}
	We say that a homeomorphism $f:X\rightarrow X$ satisfies the \emph{shadowing property} if given $\varepsilon>0$ there is $\delta>0$ such that for each sequence $(x_n)_{n\in\mathbb{Z}}\subset X$ satisfying
	$$d(f(x_n),x_{n+1})<\delta \,\,\,\,\,\, \text{for every} \,\,\,\,\,\, n\in\mathbb{Z}$$ there is $y\in X$ such that
	$$d(f^n(y),x_n)<\varepsilon \,\,\,\,\,\, \text{for every} \,\,\,\,\,\, n\in\mathbb{Z}.$$
	In this case, we say that $(x_k)_{k\in\mathbb{Z}}$ is a $\delta-$pseudo orbit of $f$ and that $(x_n)_{n\in\mathbb{Z}}$ is
	$\varepsilon-$shadowed by $y$.
\end{definition}

The following is our main result.

\begin{mainthm}\label{B} Let $f\colon X\to X$ be a homeomorphism of a compact metric space $X$ satisfying the shadowing property. 
	\begin{enumerate}
		\item If $f$ is sensitive, with sensitivity constant $\eps>0$, then for each $x\in X$ there is a compact and perfect set $$C_x\subset W^u_\varepsilon(x).$$ 
		\item If $f^{-1}$ is sensitive, with sensitivity constant $\eps>0$, then for each $x\in X$ there is a compact and perfect set $$C_x\subset W^s_\varepsilon(x).$$
	\end{enumerate}
\end{mainthm}

\begin{proof}
	This proof is inspired by the proof of Proposition 2.2.2 in \cite{ArtigueDend}. Assume that $f$ is a sensitive homeomorphism with sensitivity constant $\eps>0$. The shadowing property assures the existence of $\delta\in(0,\eps)$ such that every $\delta$-pseudo orbit of $f$ is $\varepsilon/2$-shadowed. Given $x\in X$, we can use the sensitivity of $f$ to obtain $x_1\in X$ such that $d(x, x_1)<\delta$ and 
	%$$d(f^{n}(x),f^{n}(x_1))>\varepsilon \,\,\,\,\,\, \text{for some} \,\,\,\,\,\, n\in\mathbb{N},$$ 
	$x_1\notin W^s_{\varepsilon}(x)$. Consider the sequence $(x^k)_{k\in\mathbb{Z}}$ defined as follows
	$$x^k=\left\{\begin{array}{ll}
	f^k(x), & k< 0\\
	f^k(x_1), & k\geq 0.\\
	\end{array}\right.$$ 
	The sequence $(x^k)_{k\in\mathbb{Z}}$ is a $\delta$-pseudo orbit of $f$ and then by the shadowing property there is 
	$$c_1(x)\in W^u_{\varepsilon/2}(x)\cap W^s_{\varepsilon/2}(x_1).$$
	Note that $c_1(x)\neq x$ since $c_1(x)\in W^s_{\varepsilon/2}(x_1)$ and $x_1\notin W^s_{\varepsilon}(x)$, and consider the set $$C_1=\{x, c_1(x)\}.$$
	Let $\varepsilon_1>0$ be such that $$\eps_1<\min\{\eps/4,d(x,c_1(x))/2\}$$
	and choose $\delta_1\in(0,\eps_1)$, given by the shadowing property, such that every $\delta_1$-pseudo orbit of $f$ is $\eps_1$-shadowed. We can use the sensitivity of $f$ for each $y\in C_1$ to obtain $y_1=y_1(y)$ such that $$d(y,y_1)<\delta_1 \,\,\,\,\,\, \text{and} \,\,\,\,\,\, y_1\notin W^s_\varepsilon(y).$$ The sequence $(y^k)_{k\in\mathbb{Z}}$ given by $$y^k=\left\{\begin{array}{ll}
	f^k(y), & k<0\\
	f^k(y_1), & k\geq0\\
	\end{array}\right.$$ 
	is a $\delta_1$-pseudo orbit of $f$, so the shadowing property assures the existence of 
	$$c_2(y)\in W^u_{\varepsilon_1}(y)\cap W^s_{\varepsilon_1}(y_1)$$ 
	Note that $c_2(y)\in W^u_\varepsilon(x)$ for every $y\in C_1$ since 
	$$y\in W^u_{\varepsilon/2}(x), \,\,\,\,\,\, c_2(y)\in W^u_{\varepsilon_1}(y) \,\,\,\,\,\, \text{and} \,\,\,\,\,\, \eps/2+\eps_1<\eps.$$ 
	Also, $c_2(y)\neq y$ since $c_2(y)\in W^s_{\varepsilon_1}(y_1)$ and $y_1\notin W^s_{\varepsilon}(y)$. Moreover, 
	$$c_2(y)\neq z \,\,\,\,\,\, \text{for each} \,\,\,\,\,\, z\in C_1$$ because $d(c_2(y),y)<\eps_1$ and $d(y,z)>\eps_1$ if $z\in C_1\setminus\{y\}$.
	Thus, the set $$C_2=C_1\cup \{ c_2(y); \,\,y\in C_1\}$$ has $2^2$ elements, $C_2\subset W^u_\varepsilon(x)$ and for each $y\in C_{1}$ there is $c_2(y)\in C_2$ such that $$d(c_2(y),y)<\frac{\eps}{2^2}.$$
	We can construct using an induction process an increasing sequence of sets $(C_k)_{k\in\mathbb{N}}$ such that $C_k$ has $2^k$ elements, $C_k\subset W^u_\varepsilon(x)$ and for each $y\in C_{k-1}$ there exists $c_k(y)\in C_k$ such that $$d(c_k(y),y)<\frac{\eps}{2^k}.$$
	Thus, we can consider the set $$C_x=\overline{\bigcup_{k\geq 1}C_k},$$ that is a compact set contained in $W^u_\varepsilon(x)$, since $W^u_\varepsilon(x)$ is closed and $C_k\subset W^u_\varepsilon(x)$ for every $k\in\N$. To see that $C_x$ is perfect let $z\in C_x$. If $z\notin C_k$ for every $k\in\N$, then clearly $z$ is accumulated by points of $C_x$. If $z\in C_k$ for some $k\in\N$, then 
	$$z\in C_n \,\,\,\,\,\, \text{for every} \,\,\,\,\,\, n\geq k,$$ since $(C_k)_{k\in\N}$ is an increasing sequence. Thus, for each $\alpha>0$ we can choose $N>k$ such that $$\frac{\eps}{2^N}<\alpha$$ and since $z\in C_N$ it follows that there exists $c_N(z)\in C_{N+1}$ satisfying
	$$d(c_N(z),z)<\frac{\eps}{2^N}<\alpha.$$ So, for each $z\in C_x$ and $\alpha>0$ we can find $c_N(z)\in C_x$ such that $d(z,c_N(z))<\alpha$. This proves that $z$ is an accumulation point of $C_x$ and that $C_x$ is perfect. The proof for the case $f^{-1}$ sensitive is analogous and we leave the details to the reader.
\end{proof} 

Next section we obtain consequences of this theorem generalizing results in \cite{ACCV} and \cite{CC2}.

\section{Positive cw-expansivity and shadowing}

In \cite{Kato93} and \cite{Kato93B} Kato introduced a positive notion of cw-expansiveness. 

\begin{definition}
A map $f:X\rightarrow X$ is \emph{positively cw-expansive} if there exists $c>0$ such that $W^s_c(x)$ is totally disconnected for every $x\in X$. Equivalently, for every non-trivial continuum $C\subset X$ there exists $n\in\N$ such that $$\diam(f^n(C))>c.$$
\end{definition}

Positively cw-expansive maps defined on continua (compact, connected and non-trivial) satisfy sensitivity to initial conditions. Indeed, every non-empty open set contains a non-trivial continuum that increase when iterated to the future, so local stable-sets cannot contain any open set of the space. Kato exhibit examples of positively cw-expansive homeomorphisms and proved that they cannot be defined in Peano continua (see Corollary 1.7 in \cite{Kato93B}). We will prove that the restriction to certain hyperbolic sets provide examples of positively cw-expansive homeomorphisms satisfying the shadowing property. We omit classical definitions here such as hyperbolicity, non-wandering set $\Omega(f)$, attractor, manifolds and foliations because in the proof we only need the Theorem 1 in \cite{ABD} that states a hyperbolic set contained in $\Omega(f)$ has either empty interior or is the whole ambient manifold. 

\begin{theorem}
Let $f\colon M\to M$ be a diffeomorphism defined in a manifold and let $\Lambda\subset M$ be a hyperbolic attractor of $f$. If $\Lambda \subset \Omega(f)$, $\Lambda\neq M$ and its stable (unstable) foliation is one dimensional, then $f|_{\Lambda}$ ($f^{-1}|_{\Lambda}$) is positively cw-expansive.
\end{theorem}

\begin{proof} We assume that the stable foliation has dimension one and prove that $f|_{\Lambda}$ is positively cw-expansive. By contradiction, suppose that $f|_{\Lambda}$ is not positively cw-expansive, that is, for each $\eps>0$ there exists a non-trivial continuum $C\subset\Lambda$ and $x\in \Lambda$ such that $C\subset W^s_\varepsilon(x)$. Let $$A=\bigcup_{x\in C}W^u_\varepsilon(x)$$ and note that $\Lambda$ being an attractor implies that
$$W^u_\varepsilon(x)\subset \Lambda \,\,\,\,\,\, \text{for every} \,\,\,\,\,\, x\in \Lambda.$$ This implies that $A\subset \Lambda$ since $C\subset \Lambda$. Since $C$ is a non-degenerate continuum contained in some stable set which is a one-dimensional manifold, it follows that $C$ is also a one-dimensional manifold. This ensures that the interior of $A$ is not empty. Thus the interior of $\Lambda$ is not empty, since $A\subset \Lambda$. By Theorem 1 in \cite{ABD} we have that $\Lambda=M$, contradicting the hypothesis $\Lambda\neq M$. 
\end{proof}

Surface DA attractors and the Solenoid are examples of hyperbolic attractors illustrating this result (see \cite{Robinson} for details of these attractors). All these examples are positively cw-expansive (and hence sensitive) homeomorphisms satisfying the shadowing property. They have in common that local stable sets are uncountable. Indeed, it is proved in \cite{MM} that the Hausdorff dimension of $W^s_{\eps}(x)\cap\Lambda$ is positive in the case $\Lambda$ is a basic piece of an Axiom A diffeomorphism of a surface. We now start to discuss the case all local stable sets are countable.

\begin{definition}
We say that a homeomorphism $f:X\rightarrow X$ is \emph{positively countably-expansive} if there exists $\varepsilon>0$ such that $W^s_\varepsilon(x)$ is countable for every $x\in X$. Equivalently, for every $C\subset X$ uncountable, there exists $n\in\N$ such that $$\diam(f^n(C))>\eps.$$
\end{definition}

Any homeomorphism defined in a countably and compact metric space is clearly positively countably expansive. Thus, the identity map on a countably compact space is positively countably expansive and satisfies the shadowing property (see Theorem 2.3.2 in \cite{AH}). The question we now start to discuss is:

\begin{question}
Does there exist a positively countably expansive homeomorphism with the shadowing property defined in an uncountable compact metric space?
\end{question}

We answer this question negatively in two particular cases: the first assuming transitivity and the second assuming the L-shadowing property. 

\begin{definition}
A map $f\colon X\to X$ is called \emph{transitive}, if for any pair $U,V\subset X$ of non-empty open subsets, there exists $n\in\N$ such that $$f^n(U)\cap V\neq\emptyset.$$
\end{definition}

In this case, there is a residual set of points whose future orbits are dense on the space. A point $x\in X$ is called \emph{chain-recurrent} if for each $\eps>0$ there exists a non-trivial finite $\eps$-pseudo orbit starting and ending at $x$. The set of all chain-recurrent points is called the \textit{chain recurrent set} and is denoted by $CR(f)$. This set can be split into disjoint, compact and invariant subsets, called the \emph{chain-recurrent classes}. The \emph{chain-recurrent class} of a point $x\in X$ is the set of all points $y\in X$ such that for each $\eps>0$ there exist a periodic $\eps$-pseudo orbit containing both $x$ and $y$. If $f$ is transitive, then the whole space $X$ is a chain recurrent class. Now we define the L-shadowing property.

\begin{definition}
A homeomorphism $f$ satisfies the L-shadowing property if for every $\eps>0$, there exists $\delta>0$ such that for every sequence $(x_k)_{k\in\Z}\subset X$ satisfying $$d(f(x_k),x_{k+1})\leq\delta\,\,\,\,\,\, \text{for every} \,\,\,\,\,\, k\in\Z \,\,\,\,\,\, \text{and}$$  $$d(f(x_k),x_{k+1})\to0 \,\,\,\,\,\, \text{when} \,\,\,\,\,\, |k|\to\infty,$$ there is $z\in X$ satisfying $$d(f^k(z),x_k)\leq\eps \,\,\,\,\,\, \text{for every} \,\,\,\,\,\, k\in\Z \,\,\,\,\,\, \text{and}$$ $$d(f^k(z),x_k)\to0\,\,\,\,\,\, \text{when} \,\,\,\,\,\, |k|\to\infty.$$ 
The sequence $(x_k)_{k\in\Z}$ is called a $\delta$-\emph{limit-pseudo-orbit} of $f$ and we say that $z$ $\eps$-\emph{limit-shadows} the sequence $(x_k)_{k\in\Z}$.
\end{definition}

The L-shadowing property was introduced in \cite{CC2} and further explored in \cite{ACCV} and \cite{ACCV3}. This is a stronger version of the shadowing property that is present on continuum-wise-hyperbolic systems as proved in \cite{ACCV3} and it implies a spectral decomposition of the chain recurrent set (see \cite{ACCV}). The chain-recurrent classes of homeomorphisms satisfying the L-shadowing property are either expansive or admit arbitrarilly small \emph{topological semihorseshoes} (see Theorem B in \cite{ACCV}) that are compact periodic sets whose restriction is semiconjugate to a shift of two symbols. In particular, topological semihorseshoes are uncountable sets with positive entropy contained in arbitrarilly small dynamical balls. The following is our second main result.

\begin{mainthm}\label{C}
	Let $f:X\rightarrow X$ be a positively countably-expansive homeomorphism, defined in a compact metric space $X$.
	If at least one of the following conditions is satisfied
	\begin{itemize}
		\item[(1)] $f$ is transitive and has the shadowing property
		\item[(2)] $f$ has the L-shadowing property
	\end{itemize} then $X$ is countable.
\end{mainthm}

We will split the proof of this theorem in the items (1) and (2) and before proving each item we will state a few definitions and results that are going to be important in the proof.

\begin{definition}
We say that a map $f\colon X\to X$ is \emph{equicontinuous} if the sequence of iterates $(f^n)_{n\in\N}$ is an equicontinuous sequence of maps. Equivalently, for every $\eps>0$ there exists $\delta>0$ such that $$d(x,y)<\delta \,\,\,\,\,\, \text{implies} \,\,\,\,\,\, d(f^n(x),f^n(y))<\eps \,\,\,\,\,\, \text{for every} \,\,\,\,\,\, n\in\N.$$ In this case, the $W^s_{\eps}(x)$ contains the ball of radius $\delta$ centered at $x$. 
\end{definition}

A classical result in topological dynamics is the Auslander-York dichotomy: a minimal homeomorphism of a compact metric space is either sensitive or equicontinuous. In \cite[Corollary 1]{Moot} Moothathu proved that a transitive homeomorphism satisfying the shadowing property is either sensitive or equicontinuous. We remark that a consequence of Theorem \ref{B} is the following result. 

\begin{corollary}\label{sen}
If $f:X\rightarrow X$ is a positively countably-expansive homeomorphism, defined in a compact metric space $X$, and satisfying the shadowing property, then $f^{-1}$ is not sensitive. 
\end{corollary}

\begin{proof}
If $f^{-1}$ is sensitive, Theorem \ref{B} assures the existence of compact and perfect sets on every local stable set, but this imples that local stable sets are uncountable and contradicts the hypothesis of positive countable expansivity.
\end{proof}

\begin{proof}[Proof of Theorem \ref{C} (1)]
Let $f$ be a positively countably expansive homeomorphism that is transitive and satisfies the shadowing property. Note that $f^{-1}$ is also transitive and satisfies the shadowing property, so $f^{-1}$ is either sensitive or equicontinuous. Corollary \ref{sen} assures that $f^{-1}$ cannot be sensitive, so it is equicontinuous. Since $f$ is a homeomorphism, \cite[Theorem 3.4]{AG} assures that $f$ is also equicontinuous. Let $c>0$ be a positively countably expansive constant of $f$ and choose $\delta>0$, given by equicontinuity, such that
$$d(x,y)<\delta \,\,\,\,\,\, \text{implies} \,\,\,\,\,\, y\in W^s_{c}(x).$$ Thus, every open set of diameter smaller than $\delta$ is contained in a same $c$-stable set and, hence, is countable. Since $X$ is compact, we can choose a finite open cover with elements of diameter smaller than $\delta$. This implies that $X$ is countable, since it is written as a finite union of countable sets.
\end{proof}

Now we define the limit shadowing property.

\begin{definition}
A sequence $(x_k)_{k\in\N}\subset X$ is called a \emph{limit pseudo-orbit} if it
satisfies $$d(f(x_k),x_{k+1})\rightarrow 0 \,\,\,\,\,\, \text{when} \,\,\,\,\,\, k\rightarrow\infty.$$
The sequence $\{x_k\}_{k\in\N}$ is \emph{limit-shadowed} if there exists $y\in X$ such that $$d(f^k(y),x_k)\rightarrow 0, \,\,\,\,\,\, \text{when} \,\,\,\,\,\,k\rightarrow
\infty.$$ We say that $f$ has the \emph{limit shadowing property} if every limit pseudo-orbit is limit-shadowed.
\end{definition}

This property was introduced by Eirola, Nevanlinna and Pilyugin in \cite{ENP}, see also \cite{Cthesis}, \cite{C}, \cite{C2}, \cite{CK} and \cite{P1}. The L-shadowing property refines both the shadowing and the limit shadowing properties since homeomorphisms satisfying the L-shadowing property also satisfies both shadowing and limit shadowing. Indeed, the shadowing property is proved in \cite[Proposition 2]{CC2} and the limit shadowing property can be seen as a consequence of \cite[Theorem 2.4]{ACCV}. We exhibit in the next result a much simpler proof of this fact.

\begin{proposition}
If $f\colon X\to X$ is a homeomorphism satisfying the L-shadowing property, then it satisfies the limit shadowing property.
\end{proposition}

\begin{proof}
Let $(x_k)_{k\in\N}$ be a limit pseudo orbit of $f$ and let $\eps=\diam(X)$. The L-shadowing property assures the existence of $\delta>0$ such that any $\delta$-limit pseudo orbit is $\eps$-limit shadowed. Choose $N\in\N$ such that
$$d(f(x_k),x_{k+1})<\delta \,\,\,\,\,\, \text{for every} \,\,\,\,\,\, k\geq N$$ and consider the sequence $(y_k)_{k\in\N}$ defined by
$$y_k=\begin{cases}
x_{N+k}, & k\geq0\\
f^{k}(x_N), & k<0.
\end{cases}$$
This is clearly a $\delta$-limit pseudo orbit and, hence, there exists $z\in X$ that $\eps$-limit shadows it. In particular, $f^{-N}(z)$ limit shadows $(x_k)_{k\in\N}$.
\end{proof}

In the following we use the limit shadowing property and the finiteness of the number of distinct chain recurrent classes to write the whole space as the union of the stable sets of chain recurrent points.

\begin{proposition}\label{L}
If $f\colon X\to X$ is a homeomorphism defined in a compact metric space and satisfies the L-shadowing property, then
$$X=\bigcup_{x\in CR(f)}W^s(x).$$
\end{proposition}

\begin{proof}
It is proved in \cite{CC2} that the L-shadowing property implies that the chain recurrent set is a finite union of distinct chain recurrent classes $$CR(f)=\bigcup_{i=1}^nC_i.$$ This implies that the restriction of $f$ to each of these classes satisfies the L-shadowing property and, in particular, the limit shadowing property.
The argument in \cite[Theorem 3.2.2]{AH} proves that
$$X=\bigcup_{i=1}^nW^s(C_i),$$ where $W^s(C)=\{y\in X; \lim_{k\to\infty}d(f^k(y),C)=0\}$. Hence, if $z\in X$, then 
$$z\in W^s(C_i) \,\,\,\,\,\, \text{for some} \,\,\,\,\,\, i\in\{1,\dots,n\}.$$ We can project the orbit of $z$ into the class $C_i$ considering a sequence of points $(x_k)_{k\in\N}\subset C_i$ that minimize the distance between $f^k(z)$ and $C_i$. It follows from $$\lim_{k\to\infty}d(f^k(z),C_i)=0$$ that $(x_k)_{k\in\N}$ is a limit pseudo orbit of $f$ and then the limit shadowing property assures the existence of $x\in C_i$ that limit shadows $(x_k)_{k\in\N}$. In particular we obtain
$$\lim_{k\to \infty}d(f^k(z),f^k(x))=0 \,\,\,\,\,\, \text{i.e.} \,\,\,\,\,\, z\in W^s(x).$$
\end{proof}

\begin{proof}[Proof of Theorem \ref{C} (2)]
Let $f$ be a positively countably expansive homeomorphism satisfying the L-shadowing property. In this case, there is only a finite number of distinct chain recurrent classes and the restriction of $f$ to each of these classes is transitive, has the shadowing property and, by hypothesis, is positively countably expansive. Then item (1) assures that each chain recurrent class is countable, and since there is only a finite number of them, the chain recurrent set is countable. Proposition \ref{L} ensures that 
$$X=\bigcup_{x\in CR(f)}W^s(x)$$
and, hence, to prove that $X$ is countable it is enough to prove that $W^s(x)$ is countable for every $x\in CR(f)$. 
%A consequence of \cite[Theorem B]{ACCV} is that $f|_{CR(f)}$ is expansive, because otherwise, $CR(f)$ would contain arbitrarilly small topological semihorseshoes, which as observed above, contradicts positive countable expansivity. 
Since
$$W^s(x)\subset\bigcup_{n\in\N\cup\{0\}}f^{-n}(W^s_{\eps}(f^n(x))) \,\,\,\, \text{ for every} \,\,\,\,\,\, x\in CR(f),$$ the existence of $x\in CR(f)$ such that $W^s(x)$ is uncountable implies that $$f^{-n}(W^s_{\eps}(f^n(x))$$ would be uncountable for some $n\in\N\cup\{0\}$. Consequently $W^s_{\eps}(f^n(x))$ would be uncountable, yielding contradiction.
\end{proof}

This theorem generalizes Theorems A and B in \cite{CC2} and Theorem G in \cite{ACCV} to the case of positive countable expansivity. In \cite{CC2} it is proved that positively n-expansive homeomorphisms with the additional assumptions of transitivity and shadowing, or the L-shadowing property, can only be defined on finite spaces. More generaly, in \cite{ACCV} is proved that positively finite-expansive homeomorphisms satisfying the shadowing property can only be defined in finite spaces.

\section{Examples}

In this last section we introduce more examples of sensitive homeomorphisms satisfying the shadowing property. The first class of examples considers the continuum-wise expansive homeomorphisms defined by Kato in \cite{Kato93}.

\begin{definition}
We say that $f\colon X\to X$ is \emph{continuum-wise expansive} if there exists $c>0$ such that $W^u_c(x)\cap W^s_c(x)$ is totally disconnected for every $x\in X$. Equivalently, for each non-trivial continuum $C\subset X$, there exists $n\in\Z$ such that \[\diam(f^n(C))>c.\]
The number $c>0$ is called a cw-expansivity constant of $f$ and the set $W^u_c(x)\cap W^s_c(x)$ is called the dynamical ball of $x$ and radius $c$. 
\end{definition}

Cw-expansive homeomorphisms defined on Peano continua (locally connected continua) are sensitive (see \cite{Hertz}). Thus, any cw-epansive homeomorphism satisfying the shadowing property is an example of a sensitive homeomorphism with the shadowing property. Examples of these systems can be found in \cite{ACCV3}: the pseudo-Anosov diffeomorphism of $\mathbb{S}^2$ and more generally the continuum-wise hyperbolic homeomorphisms.

We can construct sensitive homeomorphisms with the shadowing property that are not cw-expansive as follows. Let $(X,d_X)$ and $(Y,d_Y)$ be metric spaces and $X\times Y$ be the product space endowed with the metric $$d_{X\times Y}((x, y),(x', y'))=\max\{d_X(x,x'), d_Y(y,y')\}.$$ Let $f:X\rightarrow X$ and $g:Y\rightarrow Y$ be homeomorphisms and consider the product homeomorphism $f\times g:X\times Y\rightarrow X\times Y$ defined by $$f\times g\;(x,y)=(f(x),g(y))$$ for every $(x,y)\in X\times Y$.

\begin{theorem}
If $f$ and $g$ have the shadowing property and one of them is sensitive, then $f\times g$ is sensitive and has the shadowing property. Moreover, if either $f$ or $g$ is not cw-expansive then $f\times g$ is not cw-expansive.
\end{theorem}

\begin{proof} Suppose that $f$ and $g$ have the shadowing property and $f$ is sensitive. Let $\varepsilon>0$ be the sensitivity constant of $f$. Given $(x,y)\in X\times Y$ and $\delta>0$, the sensitivity of $f$ assures the existence of $x'\in X$ with $d_X(x', x)<\delta$ and $n\in\mathbb{N}$ such that
	$$d_X(f^n(x),f^n(x'))>\varepsilon.$$ Note that $$d_{X\times Y}((x',y), (x,y))=\max\{d_X(x',x), d_Y(y,y)\}=d_X(x',x)<\delta$$ and
	$$\begin{array}{rcl}
	d_{X\times Y}((f\times g)^n(x',y), (f\times g)^n(x, y))&=& d_{X\times Y}((f^n(x'),g^n(y)), (f^n(x),g^n(y)))\\
	& & \\
	&=&\max\{d_X(f^n(x'),f^n(x)), d_Y(g^n(y),g^n(y))\}\\
	& & \\
	&=&d_X(f^n(x'), f^n(x))\;>\;\varepsilon.\\
	\end{array}$$ 
	This proves that $f\times g$ is sensitive. The fact that the product of homeomorphisms which have shadowing property also has the shadowing property is known, see \cite[Theorem 2.3.5]{AH}.
	Now, suppose that $g$ is not cw-expansive, i.e., for every $\varepsilon>0$ there exists a non-degenerate continuum $C\subset Y$ such that 
$$\mbox{diam}(g^n(C))<\varepsilon \,\,\,\,\,\, \text{for every} \,\,\,\,\,\, n\in\mathbb{Z}.$$ Then for each $x\in X$, the subset $\{x\}\times C$ of $X\times Y$ is a non-degenerate continuum satisfying
	$$\mbox{diam}((f\times g)^n(\{x\}\times C))<\varepsilon \,\,\,\,\,\, \text{for every} \,\,\,\,\,\, n\in\mathbb{Z}.$$  Therefore, $f\times g$ is not cw-expansive. 
\end{proof}

\begin{example}
	Let $f:X\rightarrow X$ be an Anosov diffeomorphism and $g:Y\rightarrow Y$ a Morse-Smale diffeomorphism. By theorem above, the product homeomorphism $f\times g$ is sensitive,  has the shadowing property and is not cw-expansive, since $f$ is sensitive and has the shadowing property and $g$ has the shadowing property and is not cw-expansive.
\end{example}

We also note that the shift map $\sigma$ on $[0,1]^{\Z}$ is sensitive (see \cite[Corollary 7.5]{AIL}), satisfies the shadowing property (see \cite[Theorem 2.3.12]{AH}) and is not cw-expansive. Indeed, for each $\varepsilon>0$ and $\underline{x}=(x_i)_{i\in\mathbb{Z}}\in [0,1]^{\Z}$, the non-degenerate continuum \[C_{\underline{x}}=\prod_{i\in\mathbb{Z}}([x_i-\varepsilon,x_i+\varepsilon]\cap [0,1])\] is contained in $W^s_\varepsilon(\underline{x})\cap W^u_\varepsilon(\underline{x})$. To see this, we consider the following metric on $[0,1]^{\Z}$: for $\underline{x}=(x_i)_{i\in\mathbb{Z}},\; \underline{y}=(y_i)_{i\in\mathbb{Z}} \in X$, let
	\[d(\underline{x},\underline{y}) = \sup_{i\in\mathbb{Z}}\frac{|x_i-y_i|}{2^{|i|}}.\]
Thus, if $\underline{y}=(y_i)_{i\in\mathbb{Z}}\in C_{\underline{x}}$, then 
\[y_i\in [x_i-\varepsilon,x_i+\varepsilon] \,\,\,\,\,\, \text{for every} \,\,\,\,\,\, i\in\mathbb{Z}\] and this implies that
\begin{eqnarray*}
d(\sigma^n(\underline{x}), \sigma^n(\underline{y})) &=& \sup_{i\in\mathbb{Z}}\frac{|x_{i+n}-y_{i+n}|}{2^{|i|}}\\
&\leq&\sup_{i\in\mathbb{Z}}\frac{\varepsilon}{2^{|i|}}\\
&\leq& \varepsilon
\end{eqnarray*}
for every $n\in\mathbb{Z}$. Thus, the shift map $\sigma$ on $[0,1]^{\Z}$ is also an example of a sensitive homeomorphism with the shadowing property that is not cw-expansive.

\section*{Acknowledgements}
The second author was supported by Capes and the Alexander von Humboldt Foundation under the project number 88881.162174/2017-1, also by CNPq grant number 405916/2018-3 and by Fapemig grant number APQ-01047-18. The authors thank R\'egis Var\~ao for participating in our first virtual meetings during the preparation of the paper and Welington Cordeiro for discussions about sensitivity.

\vspace{1.5cm}
\noindent

{\em B. Carvalho}
\vspace{0.2cm}

\noindent

Departamento de Matem\'atica,

Universidade Federal de Minas Gerais - UFMG

Av. Ant\^onio Carlos, 6627 - Campus Pampulha

Belo Horizonte - MG, Brazil.
\vspace{0.2cm}

Friedrich-Schiller-Universität Jena

Fakultät für Mathematik und Informatik

Ernst-Abbe-Platz 2

07743 Jena

\vspace{0.2cm}

\email{bmcarvalho@mat.ufmg.br}

\vspace{1.0cm}
{\em M. B. Antunes}
\vspace{0.2cm}

\noindent

Instituto de Matemática, Estatística e Computação Científica,

Universidade Estadual de Campinas - UNICAMP

Rua S\'ergio Buarque de Holanda, 651, Cidade Universitária

Campinas - SP, Brasil.

\vspace{0.2cm}

\email{ra163508@ime.unicamp.br}

\vspace{1.0cm}
{\em M. S. Tacuri}
\vspace{0.2cm}

Departamento de Matem\'atica,

Universidade Federal de Minas Gerais - UFMG

Av. Ant\^onio Carlos, 6627 - Campus Pampulha

Belo Horizonte - MG, Brazil.
\vspace{0.2cm}

margoth\underline{ }.11@hotmail.com

\noindent

\end{document}